\documentclass[11pt]{amsart}
\usepackage{amsthm}
\usepackage{geometry}
\usepackage{bbold}
\newcommand{\EEQ}{\end{equation}}
\newcommand{\rfb}[1]{\mbox{\rm
   (\ref{#1})}\ifx\undefined\stillediting\else:\fbox{$#1$}\fi}

                         \newcommand{\ud}     {{\rm d}}
\newcommand{\bt}{\begin{Theorem}}
\newcommand{\et}{\end{Theorem}}
\newcommand{\br}{\begin{remark}}
\newcommand{\er}{\end{remark}}
\newcommand{\bc}{\begin{Corollary}}
\newcommand{\ec}{\end{Corollary}}
\newcommand{\el}{\end{Lemma}}
\newcommand{\bd}{\begin{definition}}
\newcommand{\ed}{\end{definition}}

\newcommand{\R}  {\mathbb{R}}

\def\b{\beta}

\def\o{\omega}

\newcommand{\mm}    {{\hbox{\hskip 0.5pt}}}

\newcommand{\bluff} {{\hbox{\raise 15pt \hbox{\mm}}}}

\newfont{\Blackboard}{msbm10 scaled 1200}

\newfont{\roma}{cmr10 scaled 1200}

\def\CC{\rm \hbox{C\kern-.56em\raise.4ex
         \hbox{$\scriptscriptstyle |$}\kern+0.5 em }}
 
\newtheorem{corollary}{Corollary}[section]
\newtheorem{definition}[corollary]{Definition}

\newtheorem{proposition}[corollary]{Proposition}
\newtheorem{remark}[corollary]{Remark}

 \numberwithin{equation}{section}
 
\newtheorem{prop}{Proposition}[section]

\newtheorem{thm}{Theorem}[section]

\def\ds{\displaystyle}
\newcommand{\re}{\mathrm{Re}}

\newcommand{\e}{\mathrm{e}}

\begin{document}
\thispagestyle{empty}
  
\title[Stability for coupled waves with locally disturbed Kelvin-Voigt damping]{Stability for coupled waves with locally disturbed Kelvin-Voigt damping}

\author{ Fathi Hassine}
\address{UR Analysis and Control of PDEs, UR 13ES64, Department of Mathematics, Faculty of Sciences of Monastir, University of Monastir, Tunisia}
\email{fathi.hassine@fsm.rnu.tn}

\author{Nadia Souayeh}
\address{UR Analysis and Control of PDEs, UR 13ES64, Department of Mathematics, Faculty of Sciences of Monastir, University of Monastir, Tunisia}
\email{nadia.souayeh@fst.utm.tn}

\begin{abstract}
We consider a coupled wave system with partial Kelvin-Voigt damping in the interval $(-1,1)$, where one wave is dissipative and the other does not. 
When the damping is effective in the whole domain $(-1,1)$ it was proven in \cite{H.P} that the energy is decreasing over the time with a rate equal to $t^{-\frac{1}{2}}$. In this paper, using the frequency domain method we show the effect of the coupling and the non smoothness of the damping coefficient on the energy decay. Actually, as expected we show the lack of exponential stability, that the semigroup loses speed and it decays polynomially with a slower rate then given in \cite{H.P}, down to zero at least as $t^{-\frac{1}{12}}$. 
\end{abstract}    


\maketitle

\tableofcontents
  
\section{Introduction}
When a vibrating source disturbs the first particular of a medium, a wave is created. This phenomena begins to travel from particle to particle along the medium, which is typically modelled by a wave equation. In order to suppress those vibrations, the most common approach is adding damping. It's more likely  to use one of two types:
\\
\hspace*{0.2cm} 1) The linear viscous damping or "external damping", it does mostly model an external frictional force, such that the auto-mobile shock absorber.
\\
\hspace*{0.2cm} 2) The Kelvin-Voigt damping, it's also called the "internal damping" or the "material damping", which is originated from  the extension or compression of the vibrating particles.
\\
In the recent years, many researchers showed interest in problems involving this kind of damping. In control theory for instance it was shown that when the Kelvin-Voigt damping coefficient is satisfying some geometrical control conditions the semigroup corresponding to this system is exponential stable (see \cite{liu-rao2,tebou}). Nonetheless, when the damping is  arbitrary localized with singular coefficient, it's not the case any more (see \cite{AHR2,liu-rao1}). Actually, in one-dimensional case we can consider the following problem
\begin{equation} \label{intro}
\left\lbrace\begin{array}{ll}
u_{tt}-[u_{x}(x,t)+b(x)u_{xt}]_x=0 &-1<x<1,\; t\geq0,
\\
u(t,-1)=u(t,1)=0& t\geq 0,
\\
u(0,x)=u_0(x) , u_t(0,x)=u_1(x) &-1\le x \le 1,
\end{array} \right.
 \end{equation}
with $b \in L^{\infty}(-1,1)$ And 
$$
b(x)= \left\{ \begin{array}{ll}
0 &\text{for }  x\in [0,1)
\\
a(x)&\text{for } x\in (-1,0).
\end{array} \right.
$$
Under the assumption that the damping coefficient has a singularity at the interface of the damped and undamped regions, and behaves like $x^{\alpha}$ near the interface, it was proven by Liu abd Zhang \cite{liu-zhang} that the semigroup corresponding to the system is polynomially or exponentially stable and the decay rate depends on the parameter $\alpha\in(0,1]$. When $\alpha=0$, Liu and Rao \cite{liu-rao1} showed that system \eqref{intro} is polynomially stable with an order equal to $2$ where few years ago Liu and Liu \cite{liu-liu} proved the lack of the exponential stability.

When dealing with systems involving quantities described by several components, pretending to control or observe all the state variables it turns out that certain systems possess an internal structure that compensates the lack of control variables. Such a phenomenon is referred to as indirect stabilization or indirect control. For instance Alabo et al. did study in \cite{F.P} the coupled waves with partial frictional damping
\begin{equation*}
\left\lbrace \begin{array}{ll}
u_{tt}-\Delta u + \alpha v =0& x \in \Omega,\; t \geq 0,
\\
v_{tt}-\Delta v + \alpha u+ \beta v_t =0& x \in \Omega,\; t \geq 0,
\end{array} \right.
\end{equation*}
subjected to Dirichlet boundary conditions. It was proven then the semigroup corresponding to  this system is not exponentially stable, but it's polynomially with the rate $t^{\frac{-1}{2}}$. 
In 2016, Oquendo and Pacheco studied the wave equation with internal coupled terms where the Kelvin-Voigt damping is global  in one equation and the second equation is conservative. Although the damping is stronger than the frictional one, they had shown that the semigroup loses speed with a slower rate of $t^{-\frac{1}{4}}$. For this kind of coupled visco-elastic models we distinguish what is called the transmission problems which have been  intensively studied by the first author, Ammari and their collaborators in \cite{AHR2,hassine1,hassine2,hassine4,hassine3,ALS} (see also \cite{ammari-niciase}) where the systems studied in these papers are the wave or the plate equation or a coupled wave-plate equation. Assuming a non smooth and singular damping coefficient it was shown in these works a uniform and a non-uniform decay rates of the energy. In this work, we examine the behaviour of a coupled waves system with a partial Kelvin-Voigt damping, namely we consider the following system where the first wave is dissipative and the second one is conservative 
\medskip
\begin{equation}\label{sc}
\left\{
\begin{array}{ll}
u_{tt}(x,t) - [u_{x}(x,t) + a(x)u_{xt} (x,t)]_x + v_{t} (x,t) = 0& (x,t) \in(-1,1) \times (0,+\infty), 
\\
v_{tt} (x,t) - c \, v_{xx} (x,t) - u_{t}(x,t) = 0& (x,t) \in (-1,1) \times (0,+\infty), 
\\
u(0,t) = v(0,t) = 0, u(1,t) = v(1,t) = 0&\forall \, t > 0, 
\\
u(x,0) = u_0(x),  u_t (x,0) = u_1(x)& \forall \, x \in (-1,1), 
\\
v(x,0) = v_0(x), v_t (x,0) = v_1(x)& \forall \, x \in (-1,1),
\end{array}
\right.
\end{equation}
where $c >0$ and $a \in L^\infty (-1,1)$ is non-negative function. In this paper we assume that the damping coefficient is piecewise function in particular we suppose that $a$ have the following form $a  = d \, \mathbb{1}_{[0,1]}$, where $d$ is a strictly positive constant. Since the damping is singular, this system can be seen as a coupling of a conservative wave equation and a transmission wave equation.

The natural energy of (u,v) solution of \eqref{sc} at an instant $t$ is given by
$$
E(t)= \frac{1}{2} \int_{-1}^{1} \left(|u_t (x,t)|^2+|v_t(x,t)|^2 +|u_{x}(x,t)|^2 + c \, |v_x(x,t)|^2 \right)\,\ud x, \forall \, t > 0.
$$
Multiplying the first equation of \eqref{sc} by $\bar{u}$, the second by $\bar{v}$, integrating over (-1,1) and then taking the real part leads to
\begin{equation*}
E^\prime (t) = -\int_{-1}^1 a(x) \, |u_{xt}(x,t)|^2 \,\ud x,\;\forall\,t > 0.
\end{equation*}
Therefore, the energy is a non-increasing function of the time variable $t$. We show the lack of the exponential stability and prove that the semigroup corresponding to this system is polynomially stable for regular initial data and with a slower rate, down to $t^{-\frac{1}{12}}$. 

This paper is organized as follows. In section \ref{CWKVWP}, we prove that system \eqref{sc} is well-posed. In section \ref{CWKVSS}, we show that the energy of the system is the strong stability. In section \ref{CWKVLES}, we prove the lack of exponential stability. In section \ref{CWKVPS}, we prove a polynomial stability decay of the energy.
\section{Well-posedness}\label{CWKVWP}
In this section, we discuss the well-posesness of the problem (\ref{sc}) using the semigroup theory.
\medskip
Let $\mathcal{H} = (H_{0}^{1}(-1.1))^2 \times( L^2(-1,1))^2$
be the Hilbert space endowed with the inner product define, 
for $U_1 =(u^1,v^1,w^1,z^1) \in \mathcal{H}$ and $ U_2 =(u^2,v^2,w^2,z^2) \in \mathcal{H}$, by
$$ 
\left\langle U_1,U_2\right\rangle _{\mathcal{H}}= \left\langle
u^1_x,u^2_x\right\rangle _{L^{2}(-1,1)} + \left\langle
\sqrt{c} v^1_x,\sqrt{c} v^2_x \right\rangle _{L^{2}(-1,1)}
+ \left\langle w^1,w^2\right\rangle _{L^{2}(-1,1)}
+\left\langle
z^1,z^2\right\rangle_{L^{2}(-1,1)}.
$$

\medskip

By setting $y(t) = (u(t),v(t),u_t(t),v_t(t))$ and $y_0=(u_0, v_0, u_1, v_1)$ we can rewrite system (\ref{sc}) as a first order differential equation as follow
\begin{equation}\label{damped}
\dot{y}(t) =  {\mathcal A} y(t), \qquad y(0) = y_0,
\end{equation}
where
$$
{\mathcal A}(u^1,v^1,u^2,v^2) = \left(u^2, v^2, \left(u_x^1 + au_x^2 \right)_x - v^2, c \, v^1_{xx} + u^2 \right),
$$
with
\begin{align*}
(u^1,v^1,u^2,v^2) \in {\mathcal D}({\mathcal A}) = \big\{(u^1,v^1,u^2,v^2) \in \mathcal{H}, \, (u^2, v^2) \in (H^1_0(-1,1))^2,
\\
v^1 \in H^2(-1,1) \cap H^1_0 (-1,1), \left(u^1_x + a u^2_x \right)_x \in L^2(-1,1) \big\}. 
\end{align*}
For the well-posedness of system (\ref{damped}) we have the following proposition:
\begin{proposition}\label{exist} 
For an initial datum $y_0 = (u_0,v_0,u_1,v_1) \in \mathcal{H}$, there exists a unique solution $y = (u,v,u_t,v_t) \in C([0,\,+\infty),\, \mathcal{H})$
to  problem (\ref{damped}). Moreover, if $y_0 \in \mathcal{D}(\mathcal{A})$, then
$$
y = (u,v,u_t,v_t) \in C([0,\,+\infty),\, \mathcal{D}(\mathcal{A}))\cap C^{1}([0,\,+\infty),\, \mathcal{H}).
$$ 
\end{proposition}
\begin{proof}
By Lumer-Phillips' theorem (see \cite{Pazy}), it suffices to show that $\mathcal{A}$ is dissipative and  maximal.
\medskip

(1) We first prove that $\mathcal{A}$ is dissipative. Take $Z = (u,v,w,z) \in \mathcal{D}(\mathcal{A})$. 
Then
\begin{align*}
\left\langle\mathcal{A}Z,Z\right\rangle_{\mathcal{H}} = \left\langle
w_x,u_x\right\rangle _{L^{2}(-1,1)} + c \, \left\langle
z_x, v_x \right\rangle _{L^{2}(-1,1)}
+ \left\langle (u_x + aw_x)_x,w\right\rangle _{L^{2}(-1,1)}
\\
+ \left\langle
c v_{xx} + w,z\right\rangle_{L^{2}(-1,1)}.
\end{align*}
By integration by parts and using the boundary conditions, it holds:

\begin{equation}\label{res1}
(\mathcal{A} Z,Z)_{\mathcal{H}} = - \left\langle a w_{x},w_{x}\right\rangle_{L^2(-1,1)}= - \int_{-1}^{1} a|w_x|^2 \, \ud x \leq 0.
\end{equation}
This shows that $\mathcal{A}$ is the dissipative.

\medskip
(2) Let us now prove that $\mathcal{A}$ is maximal, i.e., that
$\lambda I-\mathcal{A}$ is surjective for some $\lambda>0$. So, for any given $(f,g,f_1,g_1)\in\mathcal{H}$, we solve the equation $\mathcal{A}(u,v,w,z)=(f,g,f_1,g_1)$, which is recast on the following way
\begin{equation}\label{WPwave}
\left\{\begin{array}{l}
w = f 
\\
z = g   
\\
u_{xx}  + (af_x)_x = f_1 + g 
\\
c \, v_{xx}  = g_1 - f.
\end{array}\right.
\end{equation}
It is well known that by Lax-Milgram's theorem the system \eqref{WPwave} admits a unique solution $(u,v) \in H_{0}^{1}(-1,1) \times H^1_0(-1,1)$. Moreover by multiplying the second and the third lines of (\ref{WPwave}) by $\overline{u}, \overline{v}$ respectively and integrating over $(-1,1)$ and using Poincar\'e inequality and Cauchy-Schwarz inequality we find that there exists a constant $C>0$ such that
$$
\int_{-1}^1 \left(|u_x (x)|^{2} + |v_x (x)|^2 \right) \, dx \leq C \, \int_{-1}^1 \left(|f_x (x)|^{2} + |g_x (x)|^2 + |f_1(x)|^2 + |g_1(x)|^{2} \right)\, dx.
$$
It follows that $(u,v,w,z) \in\mathcal{D}(\mathcal{A})$ and we have
$$
\|(u,v,w,z)\|_{\mathcal{H}}\leq C\|(f,g,f_1,g_1)\|_{\mathcal{H}}.
$$
This imply that $0\in\rho(\mathcal{A})$ and by contraction principle, we easily get $R(\lambda\mathrm{I}-\mathcal{A})=\mathcal{H}$ for sufficient small $\lambda>0$. The density of the domain of $\mathcal{A}$ follows from \cite[Theorem 1.4.6]{Pazy}. Then thanks to Lumer-Phillips Theorem (see \cite[Theorem 1.4.3]{Pazy}), the operator $\mathcal{A}$ generates a $C_{0}$-semigroup of contractions on the Hilbert $\mathcal{H}$ denoted by $(e^{t \mathcal{A}})_{t \geq0}$. 

\medskip 
\end{proof}  
\section{Strong stability}\label{CWKVSS}
\begin{thm} \label{th1}
The semigroup $(e^{t \mathcal{A}})_{t \geq 0}$ is strongly stable in the energy space $\mathcal{H}$ i.e.,
$$
\lim_{t\to+\infty}\|e^{t \mathcal{A}}y_{0}\|=0,\qquad \forall\,y_{0}\in\mathcal{D}(\mathcal{A}).
$$
\end{thm}

\begin{proof}
To show that the semigroup $(e^{t\mathcal{A}})_{t\geq 0}$ is strongly stable we only have to prove that the intersection of $\sigma(\mathcal{A})$ with $i\mathbb{R}$ is an empty set. Since the resolvent of the operator $\mathcal{A}$ is not compact (see \cite{liu-rao2}) but $0\in\rho(\mathcal{A})$ we only need to prove that $(i\mu I-\mathcal{A})$ is a one-to-one correspondence in the energy space $\mathcal{H}$ for all $\mu\in\mathbb{R}^{*}$. The proof will be done in two steps: In the first step we prove the injective property of $(i\mu I-\mathcal{A})$ and in the second step we prove the surjective property of the same operator.
\medskip
\subparagraph*{Step 1} Let $(u,v,w,z)\in\mathcal{D}(\mathcal{A})$ such that 
\begin{equation}\label{ICwave}
\mathcal{A}(u,v,w,z) = i\mu(u,v,w,z).
\end{equation}
or equivalently,
\begin{equation}\label{waveIC}
\left\{\begin{array}{ll}
w=i\mu u&\text{in } (-1,1),
\\
z=i\mu v&\text{in } (-1,1),
\\
(u_{x} + aw_{x})_{x} - z =i\mu w&\text{in } (-1,1),
\\
c \, v_{xx} + w = i\mu z&\text{in } (-1,1), 
\\
u(-1) = u(1) =0, \, v(-1) = v(1) = 0.
\end{array}\right.
\end{equation}
Then taking the real part of the scalar product 
of \eqref{ICwave} with $(u,v,w,z)$ we get
$$
\mathrm{Re}(i\mu\|(u,v,w,z)\|_{\mathcal{H}}^{2})=\mathrm{Re}\left\langle\mathcal{A}(u,v,w,z),(u,v,w,z)\right\rangle_{\mathcal{H}} = - d \int_0^1 |w_x|^{2} dx = 0.
$$
Which implies that
\begin{equation*}
w_{x} =0 \qquad \text{ in }\,(0,1).
\end{equation*}
This implies that from the first equation \eqref{waveIC} that
\begin{equation*}
u_{x} =0 \qquad \text{ in }\,(0,1),
\end{equation*}
which means that $u$ is a constant in $(0,1)$ and since $u(1)=0$ we obtain that
\begin{equation*}
u=w=0 \qquad \text{ in }\,(0,1),
\end{equation*}
Hence, from the third and the second equation of \eqref{waveIC} one gets
\begin{equation}\label{DCwave}
u=w=v=z=0 \qquad \text{ in }\,(0,1),
\end{equation}
Using \eqref{DCwave} then \eqref{waveIC} is reduced to the following problem
\begin{equation}\label{waveIC1}
\left\{\begin{array}{ll}
w=i\mu u&\text{in } (-1,0),
\\
z=i\mu v&\text{in } (-1,0),
\\
\mu^{2}u + u_{xx} - i \mu v =0&\text{in } (-1,0),
\\
\mu^2 v + c \, v_{xx} + i \mu u = 0&\text{in } (-1,0), 
\\
u(-1) = u(0) =0, \, v(-1) = v(0) = 0.&
\end{array}\right.
\end{equation}
\medskip
Let $y=(u,v,u_x,v_x)$ and $y_x=(u_x,v_x,u_{xx},v_{xx})$ then \eqref{waveIC1} is recast as follow
\begin{equation} \label{Cwave3}
\left\lbrace \begin{array}{cc}
y_x=A_{\mu}y & \text{in} (-1,0)
\\
Y(0)=0
\end{array} \right.
\end{equation}
where  
$$
A_{\mu}=
\begin{pmatrix}
 \begin{array}{cccc}
 0 & 0 & 1  & 0
\\
0  & 0  & 0  & 1  
\\
-\mu ^2  & i\mu  & 0   & 0
\\
-i \frac{ \mu}{c}   & -\frac{ \mu^2}{c}  & 0  & 0
\end{array} 
 
\end{pmatrix}.
$$
Since $A_{\mu} $ is a bounded operator then the unique solution of \eqref{Cwave3} is $y=0$ therefore $u=v=0$ in $(-1,0)$. Moreover, from the fist and the second equation of \eqref{waveIC1} we have $w=z=0$ in $(-1,1)$. Combining all this with\eqref{DCwave}, we deduce that $u=v=w=z=0$ in $(-1,1)$. This conclude the fist part of this proof. 
\medskip
\subparagraph*{Step 2} Now given $(f,g)\in\mathcal{H}$, we solve the equation 
$$
(i\mu I-\mathcal{A})(u,v,w,z)=(f,g,f_1,g_1).
$$
Or equivalently,
\begin{equation}\label{Swave}
\left\{\begin{array}{llcc}
w=i\mu u-f \\
z =i \mu v - g \\
\mu^{2}u + u_{xx} + i\mu\,(a u_x)_x - i \mu v = (a f_x)_x -i\mu f - f_1 - g=F 
\\
\mu^2 v + c \, v_{xx} + i \mu u = - \mu g + f - g_1=G.
\end{array}\right.
\end{equation}
Let's define the operator
\begin{align*}
\begin{array}{rrrl}
 A :& (H_0^{1}(-1,1))^2 &\longrightarrow &(H^{-1}(-1,1))^2
\\
&(u,v)&\longmapsto&(-u_{xx}-i\mu (au_x)_x+i\mu v , -cv_{xx} -i \mu u).
\end{array}
\end{align*}
First we are going to show that $A$ is an  isomorphism. For this purpose we consider the two operator $\tilde{A}$ and $C$ such that
$$
\begin{array}{rrrl}
\tilde {A}: &(H_0^{1}(-1,1))^2& \longrightarrow &(H^{-1}(-1,1))^2
\\
&(u,v)&\longmapsto&(-u_{xx}-i\mu (au_x)_x , -cv_{xx}),
\end{array}
$$
and $C$ such that $A = C+ \tilde{A}$.
\medskip
It's easy to show that $\tilde{ A}$ is an isomorphism, then we could rewrite $A = \tilde{A} (Id - \tilde{ A}^{-1} (-C))$.
\medskip
To begin with, thanks to  the compact embedding
$$
 H_0^{1}(-1,1)^2\hookrightarrow L^2(-1,1)^2 \; \hbox{and} \; L^2(-1,1)^2\hookrightarrow H^{-1}(-1,1)^2,
$$
we see that $\tilde{ A}^{-1} $ is a compact operator. Secondly, it's clear that  $C$ is a bounded operator, therefore,  thanks to  Fredholm’s alternative, we only need to prove that $(Id - \tilde{A}^{-1} (-C))$ is injective.
\medskip

Let $(u,v) \in (H_0^{1}(-1,1))^2 $ such that 
$(Id - \tilde{ A}^{-1} (-C))(u,v)=0$, which implies that 
$$(\tilde{ A} -(-C))(u,v)=0.$$
Or equivalently 
\begin{equation}\label{SCwave11}
\left\{\begin{array}{ll}
u_{xx} + i\mu\,(a u_x)_x - i \mu v = 0& \text{in }(-1,1)
\\
c \, v_{xx} + i \mu u = 0& \text{in }(-1,1)
\\
u(-1)=u(1)=0,\,v(-1)=v(1)=0.
\end{array}\right.
\end{equation}
Multiplying the first equation of \eqref{SCwave11} by $\bar{u}$ and the conjugate of the second by $v$, after integration over $(-1,1)$, it follows 
$$
- \int _{-1}^{1} \vert u_x \vert ^2 \ud x +c \int _{-1}^{1} \vert v_x \vert ^2 \ud x -i \mu \int _{-1}^{1} a \vert u_x \vert ^2 \ud x =0
$$
Next, by taking the imaginary part, we can deduce that $u_x=0$  in $(0,1)$ then $u$ is constant in $(0,1)$ where with the boundary condition $u(1)=0$ we have $u=0$ in (0,1). Moreover, using the second equation of \eqref{SCwave11} we obtain $v=0$ in $(0,1)$, which implies that \eqref{SCwave11} that
\begin{equation}\label{Swave111}
\left\{\begin{array}{ll}
u_{xx}  =i \mu v  & \text{in} (-1,1)
\\
 v_{xx} = -i \frac{\mu}{c} u  & \text{in } (-1,1)
\\
u(0)=u(-1)=0,\,v(0)=v(-1)=0  
\end{array}\right.
\end{equation}
Let $y=(u,v,u_x,v_x)$ and $y_x=(u_x,v_x,u_{xx},v_{xx})$, using the trace theorem  we have:
\begin{equation*}
\left\lbrace \begin{array}{ll}
y_x=D_{\mu}y&\text{in }(-1,0)
\\
y(0)=0,&
\end{array} \right.
\end{equation*}
where  
$$D_{\mu} =
\begin{pmatrix}
 \begin{array}{cccc}
 0&0&1&0\\
0& 0& 0& 1  \\
0&i\mu& 0& 0\\
-i\frac{ \mu^2}{c}&0&0&0
\end{array} 
 
\end{pmatrix}. $$
With a same approach as in the first step, we can have the result that we are looking for (i.e. A is an isomorphism).

\medskip

Now, rewriting the third and the fourth lines of \eqref{Swave} one gets 
\begin{equation*}
(u,v)-\mu ^2 A^{-1} (u,v)=A^{-1}(F,G).
\end{equation*}
Let $(u,v) \in \ker(Id-\mu^2 A^{-1})$, i.e. $\mu ^2 (u,v)-A(u,v)=0$,  so we can see that:
\begin{equation}\label{sWAVE10}
\left\{\begin{array}{ll}
\mu^{2}u + u_{xx} + i\mu\,(a u_x)_x - i \mu v = 0&\text{in }(-1,1)
\\
\mu^2 v + c \, v_{xx} + i \mu u =0&\text{in }(-1,1).
\end{array}\right.
\end{equation}
Furthermore, multiplying the first equation of \eqref{sWAVE10} by $ \bar{u}$ and the conjugate of the second by $v$, after integration over $(-1,1)$ and taking the imaginary part, we deduce that
\begin{equation*}
\int _{-1}^{1} a \vert u_x \vert ^2 \ud x =d \int _{0}^{1} \vert u_x \vert ^2 \ud x =0.
\end{equation*}
So, we get the same system as in the first step (see \eqref{waveIC}). Thus, $\ker ( I-\mu^2 A^{-1})= \lbrace 0_{(H^{-1}(-1,1))^2} \rbrace$. 
\\
In another hand, thanks to the compact embeddings $ H_0^{1}(-1,1)^2\hookrightarrow L^2((-1,1))^2$ and $L^2(-1,1)^2\hookrightarrow H^{-1}(-1,1)^2$, we see that $A^{-1} $ is a compact operator. Now, thanks to the Fredholm's alternative, the operator ($Id-\mu^2 A^{-1}$) is bijective in $(H_0^{1}(-1,1))^2$.Finally, the equation \eqref{Swave} have a unique solution in $ H_0^{1}(-1,1)^2$.
This completes the proof.
\end{proof}
\section{Lack of exponential stability}\label{CWKVLES}
\medskip
Now, we prove the lack of  exponential stability given by the following theorem
\begin{thm}\label{th3}
The semigroup $(e^{t\mathcal{A}})_{t \geq 0}$, is not exponentially stable in the energy space provided that $c>1$ and that
\begin{equation}\label{thetaAssum}
\sin(2\sqrt{c}n\pi)\neq O(n^{-\frac{1}{2}}),
\end{equation}
\end{thm}
Noting that the assumption $c>1$ is made here just to make the calculation readable. The second assumption \eqref{thetaAssum} can be fulfilled for instance by taking $c$ such that $2\sqrt{c}$ is an integer number. To prove \eqref{th3} we mainly use the following theorem 
\begin{thm}(see \cite{huang1, pruss}) \label{3.4}
Let $e^{t\mathcal{B}}$ be a bounded $C_0$-semigroup on a Hilbert space $H$ with generator $\mathcal{B}$ such that $i \mathbb{R} \subset \rho (\mathcal{B})$. Then  $e^{t\mathcal{B}}$ is exponentially stable if and only if There exist $a > 0$ and $M > 0$, such that 
$$
\Vert e^{t\mathcal{B}}\Vert_{\mathcal{L}(\mathcal{H})}  \le M e^{-at} , \forall t \geq 0
$$
if and only if 
$$\limsup_{\omega \in \mathbb{R}, \, |\omega| \rightarrow \infty} \Vert (i\omega I-\mathcal{B})^{-1}\Vert_{\mathcal{L}(\mathcal{H})} < \infty.$$
 \end{thm}
Now, based on the Theorem \ref{3.4} we prove the Theorem \ref{th3}.
\begin{proof}
Our main objective is to show that:
\begin{equation} \label{Objective}
\Vert (\lambda I-\mathcal{A})^{-1} \Vert_{\mathcal{L}(\mathcal{H})} \text{ is unbounded on the imaginary axis}.
\end{equation}
For $n \in \mathbb{N}$ large enough let  $\lambda= \lambda_n = i\omega_{n}$, where 
\begin{equation}\label{omegan}
\omega_n =  \sqrt{\frac{8c(c+1)n^2\pi^2+2c +\sqrt{\Delta}}{4c}}\qquad\text{with}\qquad\Delta=(8c(c-1)\pi^2n^2)^2+32(c+1)(c\pi n)^2+4c^2.
\end{equation}
It's clear that  $\omega_n \longrightarrow +\infty$ and in particular we have
\begin{equation}\label{omeganBH}
\omega_{n}=\sqrt{c}\left(2n\pi+\frac{n^{-1}}{4\pi(c-1)}-\frac{cn^{-3}}{32\pi^{3}(c-1)^{3}}+o(n^{-4})\right)
\end{equation}
and
\begin{equation}\label{BHomega}
\frac{1}{\omega_{n}}=\frac{1}{2n\pi\sqrt{c}}-\frac{1}{16\sqrt{c}(c-1)(\pi n)^{3}}+o(n^{-4}).
\end{equation}

Define $(F_1 , G_1, F_2 ,G_2)$ $\in (H_0^1(0,1))^2 \times (L^2(0,1) )^2$, such that
\begin{align*}
F_1= F_1(x,n)= 0 \qquad \forall\,x \in (-1,1),
\end{align*}
 \begin{align*}
 G_1=G_1(x,n)= \left\lbrace \begin{array}{ll}
0 &\text{in }( 0,1)
\\ 
\ds g_1= \frac{\sin(2n\pi x)}{2n\pi} &\text{in } (-1,0),
\end{array} \right.
\end{align*}
\begin{align*}
F_2=F_2(x,n)=0 \qquad \forall\,x \in (-1,1),
\end{align*}
\begin{align*}
 G_2=G_2(x,n)= \left\lbrace \begin{array}{ll}
0 &\text{in } ( 0,1)
\\ 
\ds g_2=\frac{c\sin(2n\pi x)}{i \sqrt{\frac{2c}{c+1+\sqrt{(c-1)^2+\frac{4c}{\omega_n^2}}}}} &\text{in } (-1,0).
\end{array} \right.
\end{align*}
A straight forward calculation leads to 
\begin{equation}\label{fg}
\Vert(F_1,G_1,F_2,G_2)\Vert_{\mathcal{H}}^2 =\frac{1}{2}+\frac{1}{2\mu_-} \longrightarrow \frac{1}{2}\left(1+\frac{1}{\sqrt{c}}\right)\quad\text{as}\quad  n \nearrow+\infty.
\end{equation}
Our goal is to prove that $\ds\lim_{|\lambda| \rightarrow \infty} \Vert(\lambda I - \mathcal{A})^{-1} \Vert_{\mathcal{L}(\mathcal{H})} = \infty$. That's why, we  solve the resolvent equation 
\begin{equation}\label{resol}
(\lambda I-\mathcal{A})(u^1,v^1,u^2,v^2)=(F_1,G_1,F_2,G_2).
\end{equation}
\textbf{Step 1.} For all  $x \in (0,1)$ , we have 
\begin{align} \label{eq01}
\left\lbrace \begin{array}{llll}
\lambda u ^1-u^2=0
\\
\lambda v^1-v^2=0
\\
\lambda u^2-(1+\lambda d)u^1_{xx}+v^2=0
\\
\lambda v^2-cv_{xx}^1-u^2=0
\\
v^1(1)=u^1(1)=0 
\end{array}\right.
\end{align}
Let
$$
\eta_+=\frac{-\lambda (1+\lambda d-c)+\omega_n\sqrt{r}e^{i \frac{\phi}{2} }}{2(1+\lambda d)} \qquad \text{ and } \qquad \eta_-=\frac{-\lambda (1+\lambda d-c)-\omega_n\sqrt{r} e^{i \frac{\phi}{2}}}{2(1+\lambda d)} 
$$
where
\begin{equation*}
\ds r=\sqrt{a^2+b^2},\quad \cos(\phi)= \frac{a}{r}  \quad\text{ and }\quad   \sin(\phi)= \frac{b}{r},
\end{equation*}
with
\begin{align*}
\ds a&=-(1-c)^2+d^2\omega^{2}-\frac{4c}{\omega_n^{2}} 
\\
\ds b&=-2d\left((1-c)\omega_n+\frac{2c}{\omega_n}\right).
\end{align*}
It is important to note that 
\begin{equation*}
\sqrt{a}=d\omega-\frac{(c-1)^2}{2d}\omega^{-1}-\frac{(c-1)^{4}+16cd^{2}}{8d^{3}}\omega^{-3}+o(\omega^{-3}),
\end{equation*}
\begin{equation*}
\frac{b}{a}=\frac{2(c-1)}{d}+\frac{2(c-1)-4cd}{d}\omega^{-3}+o(\omega^{-4})
\end{equation*}
and
\begin{align*}
i\frac{\sqrt{r}}{d}\e^{i\frac{\phi}{2}}&=\lambda-\frac{c-1}{d}-\frac{(c-1)^{3}-d^{2}(c-1)+2cd^{2}}{d^{3}}\lambda^{-2}
\\
&+\frac{d^{2}(c-1)^{2}-(c-1)^{4}-2cd^{3}(c-1)-2cd^{2}}{d^{4}}\lambda^{-3}+o(\omega^{-3}).
\end{align*}
Then we obtain
\begin{align}\label{eta+}
\eta_{+}&=-\lambda+\frac{c}{d}-\frac{c}{d^{2}}\lambda^{-1}+\frac{(c-1)^{3}+d^{2}(c+1)+2c}{2d^{3}}\lambda^{-2}
\\
&+\frac{(c-1)^{4}-(c-1)^{3}-d^{2}(c-1)(c-2)-2c}{2d^{4}}+o(\omega^{-3})\nonumber
\end{align}
and
\begin{align}\label{eta-}
\eta_{-}&=-\frac{(c-1)^{3}+d^{2}(c+1)}{2d^{3}}\lambda^{-2}+\frac{(c-1)^{3}(2-c)+d^{2}(c-1)(c-2-2cd)}{2d^{4}}\lambda^{-3}+o(\omega^{-3})
\end{align}
\medskip
A straightforward calculation leads to 
\begin{align} \label{u eta v}
(u^1+\eta_+ v^1)_{xx}=  ( \beta_+)^2(u^1+\eta_+ v^1)
\\
(u^1+\eta_- v^1)_{xx}=  ( \beta_-)^2(u^1+\eta_- v^1),
\end{align}
where 
$$
(\beta_{\pm})^2= \frac{c \lambda ^2- \lambda \eta_{\pm} (1+\lambda d)}{c (1 + \lambda d)}.
$$
So, for $n$ large enough we get 
$$
\beta_{\pm}= \frac{\omega_n}{\sqrt{2c(1+(d\omega_n)^{2})}} \sqrt{r_{\pm}} e^{i \frac{\phi _{\pm}}{2}},
$$  
where 
\begin{equation*}
\ds r_{\pm}=\sqrt{a_{\pm}^2 + b_{\pm}^2},\quad \cos(\phi_{\pm})= \frac{a_{\pm}}{ r_{\pm}}\quad \text{ and }\quad \sin(\phi_{\pm})= \frac{b_{\pm}}{ r_{\pm}}
\end{equation*}
with
\begin{align*} 
\ds a_{\pm}&=-(1+c)-(d\omega_n)^2\pm\sqrt{r}\left(-d\omega_n\cos\left(\frac{\phi}{2}\right)+\sin\left(\frac{\phi}{2}\right)\right)
\\
\ds b_{\pm}&=cd\omega_n \pm \sqrt{r}\left(-\cos\left(\frac{\phi}{2}\right)-d\omega_n\sin\left(\frac{\phi}{2}\right)\right).
\end{align*}
Noting that
$$
|a_{+}|=2(d\omega)^{2}+\frac{c^{2}-3c+6}{2}+o(\omega^{-1}),
$$
$$
\sqrt{|a_{+}|}=\sqrt{2}d\omega+\frac{c^{2}-3c+6}{4\sqrt{2}d}\omega^{-1}+o(\omega^{-1}),
$$
$$
b_{+}=\left(d(2cd+1-c)+\frac{(c-1)^{3}}{d}\right)\omega^{-1}+o(\omega^{-1}),
$$
and
$$
\frac{b_{+}}{a_{+}}=o(\omega^{-2}).
$$
Then we obtain
\begin{equation}\label{beta+}
\beta_{+}=\frac{\lambda}{\sqrt{c}}-\frac{(c-1)(c-2)}{8\sqrt{2}d^{2}}+o(\omega^{-1}),
\end{equation}
and
\begin{equation}\label{beta+2}
\beta_{+}^{2}=\frac{\lambda^{2}}{c}+o(1).
\end{equation}
Similarly we have
$$
b_{-}=2cd\omega+\left(d(c-1-2cd)-\frac{(c-1)^{3}}{d}\right)\omega^{-1}+o(\omega^{-1}),
$$
$$
\sqrt{b_{-}}=\sqrt{2cd\omega}\left(1+\left(\frac{c-1-2cd}{4c}-\frac{(c-1)^{3}}{4cd^{2}}\right)\omega^{-2}\right)+o(\omega^{-2})
$$
$$
a_{-}=-2c+o(\omega^{-1}),
$$
and
$$
\frac{a_{-}}{b_{-}}=-\frac{\omega^{-1}}{d}+o(\omega^{-2}),
$$
then consequently we obtain
\begin{equation}\label{beta-}
\beta_{-}=\sqrt{\frac{\omega}{d}}\e^{\frac{i\pi}{4}}-\frac{\omega^{-\frac{1}{2}}}{2d^{\frac{3}{2}}}\e^{-i\frac{\pi}{4}}+o(\omega^{-1}),
\end{equation}
and
\begin{align}\label{beta-2}
\beta_{-}^{2}&=\frac{\lambda}{d}-\frac{1}{d^{2}}+\frac{(c-1)^{3}+d(c-1)+2c}{2cd^{3}}\lambda^{-1}
\\
&-\frac{(c-1)^{3}(2-c)+d(c-1)(c-2-2cd)+2c}{2cd^{4}}\lambda^{-2}+o(\omega^{-2})\nonumber
\end{align}
Next, from  \eqref{u eta v}, we get 
\begin{equation*}
(u^1+\eta_+v^1)= c_1 e^{x\beta_+}+c_2 e^{-x \beta_+}
\end{equation*}
and
\begin{equation*}
(u^1+\eta_- v^1)= c_3 e^{x \beta_-}+c_4 e^{-x \beta_-}.
\end{equation*}
Recalling that $u ^1(1)=v^1(1)=0$ we can rewrite the last two equations as follow
\begin{equation} \label{eta-2}
(u^1+\eta_+ v^1)= c_1( e^{x\beta_+}- e^{(2-x)\beta_+}),
\end{equation}  
\begin{equation} \label{eta+2}
(u^1+\eta_- v^1)= c_3(e^{x\beta_-}- e^{(2-x)\beta_-}).
\end{equation} 
Hence by combining \eqref{eta-2} and \eqref{eta+2} we obtain
\begin{equation}\label{solu1+}
u^{1}(x)=-\frac{c_{1}\eta_{-}}{\eta_{+}-\eta_{-}}\left(\e^{\beta_{+}x}-\e^{\beta_{+}(2-x)}\right)+\frac{c_{3}\eta_{+}}{\eta_{+}-\eta_{-}}\left(\e^{\beta_{-}x}-\e^{\beta_{-}(2-x)}\right),
\end{equation}
and
\begin{equation}\label{solv1+}
v^{1}(x)=\frac{c_{1}}{\eta_{+}-\eta_{-}}\left(\e^{\beta_{+}x}-\e^{\beta_{+}(2-x)}\right)-\frac{c_{3}}{\eta_{+}-\eta_{-}}\left(\e^{\beta_{-}x}-\e^{\beta_{-}(2-x)}\right).
\end{equation}
\textbf{Step 2.} For all $x \in (-1,0)$ we have
\begin{align}\label{eq(-1.0)}
\left\lbrace \begin{array}{llll}
\lambda u ^1-u^2=0
\\
\lambda v^1-v^2=g_1
\\
\lambda u^2-u^1_{xx}+v^2=0
\\
\lambda v^2-cv_{xx}^1-u^2=g_2
\\
v^1(-1)=u^1(-1)=0. 
\end{array}\right.
\end{align}
 \medskip
Following to the third and the fourth equation of \eqref{eq01} and of \eqref{eq(-1.0)} we can deduce, thanks to the regularity of the stats, that
\begin{align} 
(1+ \lambda d)u_{x}^{1}(0^+)&=u_{x}^{1}(0^-) \label{cond1u0},
\\
v_{x}^{1}(0^+)&=v_{x}^{1}(0^-)\label{cond1v0}.
\end{align}
and
\begin{align}
(1+ \lambda d)u_{xx}^{1}(0^+)&=u_{xx}^{1}(0^-),\label{cond2u0}
\\
v_{xx}^{1}(0^+)&=v_{xx}^{1}(0^-)\label{cond2v0}.
\end{align}
We denote by  
\begin{equation}\label{alpha+}
\alpha_+= \frac{\lambda}{2}\left(c-1+\sqrt{(1-c) ^2+\frac{4c}{\o_n^2}}\right)=(c-1)\lambda-\frac{c}{c-1}\lambda^{-1}-\frac{c^{2}}{(c-1)^{3}}+o(\omega^{-3}),
\end{equation}
and
\begin{equation}\label{alpha-}
\alpha_-= \frac{\lambda}{2}\left(c-1-\sqrt{(1-c)^2+\frac{4c}{\o_n^2}}\right)=\frac{c}{c-1}\lambda^{-1}+o(\omega^{-1})
\end{equation}
and we define for $n$ large enough $\mu_{\pm}$ as follow
$$
\mu_{\pm}= \frac{\sqrt{2c}}{\sqrt{c+1-\left(\pm \sqrt{(c-1)^2+\frac{4c}{\omega _n ^2}}\right)}},
$$
in particular with the chose of $\omega_{n}$ in \eqref{omegan} one get 
$$
\ds\mu_{\pm}^2= \frac{\lambda}{\lambda-\frac{\alpha_{\pm}}{c}}.
$$
Besides, we have
\begin{equation}\label{mu+}
\mu_{+}=\sqrt{c}\left(1-\frac{c}{2(c-1)}\lambda^{-2}+o(\omega^{-2})\right),
\end{equation}
\begin{equation}\label{mu-}
\mu_{-}=1+\frac{\lambda^{-2}}{2(c-1)}+o(\omega^{-2}),
\end{equation}
and
\begin{equation}\label{mu+/mu-}
\frac{\mu_{+}}{\mu_{-}}=\sqrt{c}\left(1-\frac{c+1}{2(c-1)}\lambda^{-2}+o(\omega^{-2})\right).
\end{equation}
We set
\begin{align}
\omega^+_1(x)&= (u^2+\alpha_+v^2+  \mu_+(u^1_x + \alpha_+v^1_x),\label{omega1+}
\\ 
\omega ^-_1(x)&=(u^2+\alpha_+v^2-\mu_+(u^1_x + \alpha_+v^1_x)),\label{omega1-}
\\
\omega ^+_2(x)&= (u^2+\alpha_ -v^2 + \mu_-(u^1_x + \alpha_- v^1_x)),\label{omega2+}
\\
\omega ^-_2(x)&= (u^2+\alpha_- v ^2-\mu_-(u^1_x + \alpha_- v^1_x)).\label{omega2-}
\end{align}
Now, define $Y=(\omega^+_1 , \omega^-_1,\omega^+_2,\omega_2^-)^t$ and $Z=(g_{1x},g_2)^t$. Then we have
\begin{equation}\label{w}
Y_x = AY + BZ 
\end{equation}
where 
$$
A= \left( \begin{array}{cccc}
\mu_+( \lambda -\frac{\alpha_+}{c})  & 0&0&0
\\
0&\mu_+(-\lambda +\frac{\alpha_+}{c}) &0&0
\\
0&0& \mu_-(\lambda -\frac{\alpha_-}{c})&0
\\
0&0&0&\mu_-(- \lambda +\frac{\alpha_-}{c} )
\end{array} \right)
$$
and 
$$
B=\left(  \begin{array}{cc}
-\alpha_+ & - \mu_+ \frac{\alpha_+}{c}
\\
-\alpha_+ & \mu_+ \frac{\alpha_+}{c}
\\
-\alpha_- & - \mu_- \frac{\alpha_-}{c}
\\
-\alpha_- & \mu_- \frac{\alpha_-}{c}
\end{array}\right).
$$
Then, a straightforward calculation leads to: 
$$
\mu_+ ( \lambda -\frac{\alpha_+}{c}) = 2in\pi.
$$
Using the boundary condition at $-1$ we get
\begin{equation}\label{LES1}
\omega_1^+(-1)=-\omega_1^-(-1) \; \hbox{and} \; \omega_2^+(-1)=-\omega_2^-(-1),
\end{equation}
Taking into account of \eqref{LES1} then the solution of \eqref{w} is written as follow
\begin{align}
\o_1^+(x)&=\o_1^+(-1)\e^{2in\pi x}-\frac{\alpha_{+}}{2}\left[\left(1-\frac{\mu^{+}}{\mu^{-}}\right)(x+1)\e^{2in\pi x}+\frac{1}{2n\pi}\left(1+\frac{\mu^{+}}{\mu^{-}}\right)\sin(2n\pi  x)\right],\label{sol1+}
\\
\o_1^-(x)&=-\o_1^+(-1)\e^{-2in\pi x}-\frac{\alpha_{+}}{2}\left[\left(1-\frac{\mu^{+}}{\mu^{-}}\right)(x+1)\e^{-2in\pi x}+\frac{1}{2n\pi}\left(1+\frac{\mu^{+}}{\mu^{-}}\right)\sin(2n\pi  x)\right],\label{sol1-}
\\
\o_2^+(x)&=\o_2^+(-1)\e^{\mu_{-}\left(\lambda-\frac{\alpha_{-}}{c}\right)(x+1)}+\frac{\alpha_{-}}{2in\pi+\mu_{-}\left(\lambda-\frac{\alpha_{-}}{c}\right)}\left[\e^{-2in\pi x}+\e^{\mu_{-}\left(\lambda-\frac{\alpha_{-}}{c}\right)(x+1)}\right],\label{sol2+}
\\
\o_2^-(x)&=-\o_2^+(-1)\e^{-\mu_{-}\left(\lambda-\frac{\alpha_{-}}{c}\right)(x+1)}-\frac{\alpha_{-}}{2in\pi+\mu_{-}\left(\lambda-\frac{\alpha_{-}}{c}\right)}\left[\e^{2in\pi x}+\e^{-\mu_{-}\left(\lambda-\frac{\alpha_{-}}{c}\right)(x+1)}\right].\label{sol2-}
\end{align}
Taking the trace of $\o_1^+$ and $\o_1^-$ in \eqref{sol1+}-\eqref{sol1-} and  in \eqref{omega1+}-\eqref{omega1-} on the boundary $0$ and using the continuity of the states $u_{2}$ and $v_{2}$ we obtain
 \begin{align*}
(\omega_1^+ + \omega_1^-)(0^-)&=\alpha_+\left(\frac{\mu_+}{\mu_-}-1\right)= 2u^2(0^-)+2\alpha_+v^2(0^{-})\nonumber
\\
&= 2\lambda( u^1(0^-) + \alpha_+ v^1(0^-))=2\lambda( u^1(0^+) + \alpha_+ v^1(0^+))\nonumber
\\
&=\frac{2\lambda}{\eta_+-\eta_-}\left(c_1(1-e^{2\b_+})(\alpha_+-\eta_-)+ c_3(1-e^{2 \b_-})(\eta_+-\alpha_+)\right),
\end{align*}
where we have used the the expressions of $u^{1}$ and $v^{1}$ in \eqref{solu1+} and \eqref{solv1+}.
 \medskip
This implies that
\begin{equation}\label{c3=f(c1)}
c_3=\frac{1-e^{2\b_+}}{1-e^{2\b_-}}A_nc_1+\frac{B_n}{1-e^{2\b_{-}}}
\end{equation}
where
\begin{align} \label{An}
A_n= \frac{\eta_{-}-\alpha_{+}}{\eta_{+}-\alpha_{+}}&=\frac{c-1}{c}\left(1+\frac{\lambda^{-1}}{d}-\frac{\lambda^{-2}}{c-1}+o(\omega^{-2})\right)\nonumber
\\
&=\frac{c-1}{c}\left(1+\frac{n^{-1}}{2i\pi d\sqrt{c}}-\frac{n^{-2}}{4\pi^{2}c(c-1)}+o(\omega^{-2})\right),
\end{align}
and
\begin{align}\label{Bn}
B_n&=\frac{\alpha_{+}(\eta_{+}-\eta_{-})\left(\frac{\mu_{+}}{\mu_{-}}-1\right)}{2\lambda(\eta_{+}-\alpha_{+})}\nonumber
\\
&=\frac{(c-1)(\sqrt{c}-1)}{2c}\left(1-\frac{c-1}{d}\lambda^{-1}-\left(\frac{1}{(c-1)^{2}}+\frac{\sqrt{c}(c+1)}{2(\sqrt{c}-1)(c-1)}\right)\lambda^{-2}+o(\omega^{-2})\right)\nonumber
\\
&=\frac{(c-1)(\sqrt{c}-1)}{2c}\left(1-\frac{c-1}{2i\pi d\sqrt{c}}n^{-1}-\left(\frac{1}{(c-1)^{2}}+\frac{\sqrt{c}(c+1)}{2(\sqrt{c}-1)(c-1)}\right)\times\frac{n^{-2}}{4\pi^{2}c}+o(n^{-2})\right).
\end{align}
where we used here \eqref{eta+}, \eqref{eta-}, \eqref{alpha+}, \eqref{alpha-}, \eqref{mu+/mu-} and \eqref{omegan}.
\\
Using \eqref{sol1+}-\eqref{sol1-} and \eqref{cond2u0}-\eqref{cond2v0}, one gets
\begin{align*}
(\o_1^+-\o_1^-)'(0^-)&=2i n \pi \alpha_+\left(\frac{\mu_+}{\mu_-}-1\right)=2 \mu_+ (u^1+ \alpha_+ v^1)_{xx}(0^-)=2 \mu_+ ((1+\lambda d )u^1+ \alpha_+ v^1)_{xx}(0^+)
\\
&=\frac{2\mu_{+}\left[c_{1}\beta_{+}^{2}(1-\e^{2\beta_{+}})(\alpha_{+}-(1+\lambda d)\eta_{-})+c_{3}\beta_{-}^{2}(1-\e^{2\beta_{-}})((1+\lambda d)\eta_{+}-\alpha_{+})\right]}{\eta_{+}-\eta_{-}}.
\end{align*}
Then we obtain
\begin{equation}\label{c1=f(c3)}
c_{1}=\frac{1-e^{2\b_-}}{1-e^{2\b_+}}A_{n}'c_{3}+\frac{B_{n}'}{1-e^{2\b_{+}}}
\end{equation}
where
\begin{align}\label{An'}
A_{n}'&=\frac{\beta_{-}^{2}(\alpha_{+}-(1+\lambda d)\eta_{+})}{\beta_{+}^{2}(\alpha_{+}-(1+\lambda d)\eta_{-})}\nonumber
\\
&=\frac{c}{c-1}\left(1-\frac{\lambda^{-1}}{d}+\left(\frac{(c-1)^{3}}{2cd^{2}}+\frac{c-1}{2cd}+\frac{3-c}{2d^{2}}+\frac{1}{2}\right)\lambda^{-2}+o(\omega^{-2})\right)\nonumber
\\
&=\frac{c}{c-1}\left(1-\frac{n^{-1}}{2i\pi d\sqrt{c}}+\left(\frac{(c-1)^{3}}{2cd^{2}}+\frac{c-1}{2cd}+\frac{3-c}{2d^{2}}+\frac{1}{2}\right)\frac{n^{-2}}{4\pi^{2}c}+o(n^{-2})\right),
\end{align}
and
\begin{align}\label{Bn'}
B_{n}'&=\frac{in\pi\alpha_{+}(\eta_{+}-\eta_{-})\left(\frac{\mu_{+}}{\mu_{-}}-1\right)}{\mu_{+}\beta_{+}^{2}(\alpha_{+}-(1+\lambda d)\eta_{-})}\nonumber
\\
&=\frac{n\pi(c-\sqrt{c})}{2\omega}\left(-1+\frac{c}{d}\lambda^{-1}+\left(\frac{c+\sqrt{c}+3}{2(c-1)^2}-\frac{c+1+d^{2}}{2d^{2}}\right)\lambda^{-2}\right)+o(\omega^{-2})\nonumber
\\
&=\frac{\sqrt{c}-1}{2}\left(-1+\frac{\sqrt{c}}{2i\pi d}n^{-1}+\left(\frac{\sqrt{c}+4}{c-1}-\frac{c+1+d}{d^{2}}\right)\frac{n^{-2}}{8c\pi^{2}}+o(n^{-2})\right).
\end{align}
where we used here \eqref{eta+}, \eqref{eta-}, \eqref{beta+2}, \eqref{beta-2}, \eqref{alpha+}, \eqref{alpha-}, \eqref{mu+}, \eqref{mu+/mu-} and \eqref{omegan}.
\\
Combining \eqref{c3=f(c1)} and \eqref{c1=f(c3)} then we find that
\begin{equation}\label{c1}
c_{1}=\frac{1}{1-e^{2\beta_{+}}}\times\frac{A_{n}'B_{n}+B_{n}'}{1-A_{n}A_{n}'}=\frac{c_{1}'}{1-e^{2\beta_{+}}}
\end{equation}
and
\begin{equation}\label{c3}
c_{3}=\frac{1}{1-e^{2\beta_{-}}}\times\frac{A_{n}B_{n}'+B_{n}}{1-A_{n}A_{n}'}=\frac{c_{3}'}{1-e^{2\beta_{-}}},
\end{equation}
where following to \eqref{An}, \eqref{Bn}, \eqref{An'} and \eqref{Bn'} we have
\begin{equation}\label{c'}
c_{1}'=O(1)\qquad \text{ and }\qquad c_{3}'=O(1).
\end{equation}
\medskip
In another hand, by denoting $\theta=-i\mu_{-}\left(\lambda-\frac{\alpha_{-}}{c}\right)$ and by using the same argument as previously, one gets
\begin{align*}
(\o_2^++\o_2^-)(0^-)&= 2i\sin(\theta)\o_2^+(-1)+\frac{2 \alpha_-}{2n \pi - \theta} \sin(\theta)=2 \lambda(u^1+\alpha_- v^1)(0^-)=2 \lambda(u^1+\alpha_- v^1)(0^+)
\\
&=\frac{2\lambda}{\eta_+-\eta_-}\left(c_{1}'(\alpha_{-}-\eta_{-})+ c_{3}'(\eta_+-\alpha_-)\right).
\end{align*}
It's clear that $\theta \neq 0[\pi]$ then we can write 
\begin{align}\label{omega2}
\o_2^+(-1)&= \frac{\lambda}{i\sin(\theta)(\eta_{+}-\eta_{-})}\left[c_{1}'(\alpha_--\eta_-)+ c_{3}'(\eta_+-\alpha_-)\right]-\frac{\alpha_{-}}{2in\pi+i\theta}.
\end{align}
Noting that from \eqref{omegan}, \eqref{omeganBH}, \eqref{alpha-} and \eqref{mu-} we have
\begin{equation}\label{theta}
\theta=\omega\left(1-\frac{3}{2(c-1)}\omega^{-2}+o(\omega^{-2})\right)=\sqrt{c}\left(2n\pi+\frac{c\pi-12}{4c\pi^{2}(c-1)}n^{-1}+o(n^{-1})\right)
\end{equation}
Then from \eqref{omeganBH}, \eqref{eta+}, \eqref{eta-}, \eqref{alpha-}, \eqref{mu-} and \eqref{theta} we deduce that
\begin{equation}\label{omega2BH}
\o_{2}^{+}(-1)\sim \frac{2\pi n\sqrt{c}\,c_{3}'}{\sin(\theta)}
\end{equation}
\medskip
Using \eqref{cond1u0}-\eqref{cond1v0}, \eqref{omega1+}-\eqref{omega1-} and \eqref{sol1+}-\eqref{sol1-} we get
\begin{align}\label{omega1}
\o_{1}^{+}(-1)&=\frac{(\o_1^+-\o_1^-)(0^-)}{2}=\mu_+ (u^1+ \alpha_+ v^1)_{x}(0^-)=\mu_+ ((1+\lambda d )u^1+ \alpha_+ v^1)_{x}(0^+)\nonumber
\\
&=\frac{\mu_{+}}{\eta_{+}-\eta_{-}}\left[c_{1}\beta_{+}(1+e^{2\beta_{+}})(\alpha_{+}-(1+\lambda d)\eta_{-})+c_{3}\beta_{-}(1+e^{2\beta_{-}})((1+\lambda d)\eta_{+}-\alpha_{+})\right].
\end{align}
Then from  \eqref{omeganBH}, \eqref{eta+}, \eqref{eta-}, \eqref{beta+}, \eqref{beta-}, \eqref{alpha-} and \eqref{mu+} we deduce that
\begin{equation}\label{omega1BH}
\o_{1}^{+}(-1)\sim c_{3}'\sqrt{\frac{c}{d}}\e^{-i\frac{\pi}{4}}(2\pi\sqrt{c}n)^{\frac{3}{2}}.
\end{equation}
Next, for all $x \in (-1,0)$ we have
\begin{align}\label{v1x}
v_x^{1}(x)&=\frac{1}{2 \mu_- \mu_+ (\alpha_+-\alpha_-)}\left[\alpha_{-}(\omega_1^+(x)-\omega_1^-(x))-\alpha_{+}(\omega_2^+(x)-\omega_2^-(x))\right]\nonumber
\\
&=\frac{1}{2 \mu_- \mu_+ (\alpha_+-\alpha_-)}\Bigg[\mu_{-}\left(2\omega_{1}^{+}(-1)\cos(2n\pi x)-i\alpha_{+}\left(1-\frac{\mu_{+}}{\mu_{-}}\right)(x+1)\sin(2n\pi x)\right)
\\
&-\mu_{+}\bigg(2\omega_{2}^{+}(-1)\cos\left(\theta(x+1)\right)+\frac{2\alpha_{-}}{2in\pi+i\theta}\left(\cos(2n\pi x)+\cos\left(\theta(x+1)\right)\right)\bigg)\Bigg],\nonumber
\end{align}
where we have used \eqref{omega1+}-\eqref{omega2-} and \eqref{sol1+}-\eqref{sol2-}. Thus further leads to
\begin{align}\label{v1xN}
\|v_{x}^{1}\|_{L^{2}(-1,0)}^{2}\geq\max\left\{\frac{|\omega_{1}^{+}(-1)|^{2}}{2\mu_{+}^{2}|\alpha_{+}-\alpha_{-}|^{2}},\frac{|\omega_{2}^{+}(-1)|^{2}}{\mu_{-}^{2} |\alpha_{+}-\alpha_{-}|^{2}}\right\}-\frac{|\alpha_{+}|^{2}(\mu_{+}-\mu_{-})^{2}}{4\mu_{-}^{2}\mu_{+}^{2} |\alpha_{+}-\alpha_{-}|^{2}}
\\
-\min\left\{\frac{|\omega_{1}^{+}(-1)|^{2}}{2\mu_{+}^{2}|\alpha_{+}-\alpha_{-}|^{2}},\frac{|\omega_{2}^{+}(-1)|^{2}}{\mu_{-}^{2} |\alpha_{+}-\alpha_{-}|^{2}}\right\}-\frac{2|\alpha_{-}|^{2}}{\mu_{-}^{2}\mu_{+}^{2}(2n\pi+\theta)^{2}|\alpha_{+}-\alpha_{-}|^{2}}.\nonumber
\end{align}
Since,
$$
\sin(\theta)\neq O(n^{-\frac{1}{2}}),
$$
as $n$ goes to the infinity (by \eqref{theta} assumption \eqref{thetaAssum}) then by using \eqref{omegan}, \eqref{mu+}, \eqref{mu-}, \eqref{alpha+}, \eqref{alpha-}, \eqref{omega2BH} and \eqref{omega1BH} we can show that the second and the fourth terms of the right hand side of \eqref{v1xN} are bounded while the sum of the fist and the third terms  tends to the infinity as $n$ goes to $+\infty$, therefore we obtain 
\begin{equation}\label{vx1BH}
\|v_{x}^{1}\|_{L^{2}(-1,0)}^{2}\,\,\quad\text{as}\quad n\nearrow+\infty.
\end{equation}
Last but not least, we have
 \begin{equation} \label{fin}
 \Vert (i \omega_n I- \mathcal{A})^{-1}(F_1,G_1,F_2,G_2)\Vert_{\mathcal{H}}=\Vert (u^1,v^1,u^2,v^2)\Vert_{\mathcal{H}}^2\geq 
\int_{-1}^0 \vert v_{x}^{1}(x) \vert^2 \ud x \longrightarrow+\infty,\quad\text{as}\quad n\nearrow+\infty.
\end{equation}
Finally we conclude, using \eqref{fin} and \eqref{fg} that 
$$
\limsup_{\omega \in \mathbb{R}, |\omega| \rightarrow \infty}\Vert ( i\omega I- \mathcal{A})^{-1}\Vert_{\mathcal{L}(\mathcal{H})} =+\infty.
$$ 
So, $e^{t\mathcal{A}}$ is not exponentially stable in the energy space. This completes the proof.
\end{proof}
\section{Polynomial stabilization}\label{CWKVPS}
\medskip
This subsection aims to prove the polynomial stability given by the following theorem:
\begin{thm} \label{th2}
The semigroup of contraction $(e^{T\mathcal{A}})_{t\geq 0}$ is polynomially stable of order $\ds\frac{1}{12}$.
\end{thm}
 Our method is based on the Borichev and Tomilov result given by the following:
\begin{thm}\label{SMKV19}\cite[Theorem 2.4]{BT}
Let $\mathcal{B}$ be a generator of a $C_{0}$-semigroup of contraction in a Hilbert space $\mathcal{X}$ with domain $\mathcal{D}(\mathcal{B})$ such that $i\R\subset\sigma(\mathcal{B})$ then $e^{t\mathcal{B}} $ is polynomially stable with order $\frac{1}{\gamma}, \gamma > 0$ i.e. there exists $C>0$ such that
$$
\|e^{t\mathcal{B}}U_{0}\|_{\mathcal{X}}\leq\frac{C}{(1+t)^{\frac{1}{\gamma}}}\|U_{0}\|_{\mathcal{D}(\mathcal{B})},\qquad\forall\,t\geq0,\;\forall\,U_{0}\in\mathcal{D}(\mathcal{B}),
$$
if and only if
$$
\limsup_{\beta\in\R, \, |\beta| \rightarrow\infty}\|\beta^{-\gamma}(i\beta-\mathcal{B})^{-1}\|_{\mathcal{L}(\mathcal{X})}<+\infty.
$$
\end{thm}
Based on Theorem \ref{SMKV19} we are able now to prove our main result given in Theorem \ref{th2} of this section. For this purpose, let's consider the following:
\begin{prop} The operator $\mathcal{A}$ defined in \eqref{damped} satisfies:
\begin{equation}\label{SMKV16}
\limsup_{\beta\in\R, \, |\beta| \rightarrow \infty}\|\beta^{-12}(i\beta-\mathcal{A})^{-1}\|_{\mathcal{L}(\mathcal{H})}<+\infty.
\end{equation}
\end{prop}
\begin{proof}
To prove \eqref{SMKV16}  we use an argument of contradiction. In fact, if \eqref{SMKV16} is false, then, there exist $\beta_{n}\in \R_{+}$ and $Y_{n}=(u_n^1,v_n ^1,u_n^2,v_n^2)\in\mathcal{D}(\mathcal{A})$ such that 
\begin{equation}\label{SMKV4}
\|Y_n \|_{\mathbf{H}} =1,\;\beta_{n}\nearrow +\infty\text{ and }\beta^{\gamma} (i\beta_{n}\mathrm{I}-\mathcal{A})Y_{n} :=(f_{n}^{1},g_{n}^{1},f_{n}^{2},g_{n}^{2})\longrightarrow\, 0\text{ in }\mathcal{H}\text{ as } n\nearrow+\infty.
\end{equation}
Equivalently, we have  
\begin{equation} \label{eq1}
\beta_{n}^{\gamma}\left(i\beta_{n}u_{n}^{1}-u_{n}^{2}\right)=f_{n}^{1}\,\longrightarrow 0 \hspace{0.3cm} \text{in} \hspace{0.3cm} H_{0}^{1}(-1,1),
\end{equation}
\begin{equation} \label{eq2}
\beta_{n}^{\gamma}\left(i\beta_{n}v_{n}^{1}-v_{n}^{2}\right)=g_{n}^{1}\,\longrightarrow 0 \hspace{0.3cm} \text{in} \hspace{0.3cm} H_{0}^{1}(-1,1),
\end{equation}
 \begin{equation} \label{eq3}
\beta_{n}^{\gamma}\left(i\beta_{n}u_{n}^{2}-\left(u_{nx}^{1}+au_{nx}^{2}\right)_{x}+v_{n}^{2}\right)=f_{n}^{2}\,\longrightarrow 0 \hspace{0.3cm} \text{in} \hspace{0.3cm} L^{2}(-1,1),
\end{equation}
\begin{equation}\label{eq4}
\beta_{n}^{\gamma}\left(i\beta_{n}v_{n}^{2}-cv_{nxx}^{1}-u_{n}^{2}\right)=g_{n}^{2}\,\longrightarrow 0 \hspace{0.3cm} \text{in}\hspace{0.3cm}L^{2}(-1,1).
\end{equation}
We denote by 
$$
T_{n}=u_{nx}^{1}+au_{nx}^{2}.
$$
Taking the real part of $\ds\left\langle\beta^{\gamma}(i\beta_{n}\mathrm{I}-\mathcal{A})Y_{n},Y_{n}\right\rangle_{\mathcal{H}}$ then by the dissipation property of the semigroup of the operator $\mathcal{A}$ we get
$$
\beta_{n}^{\gamma} \int_{0}^{1}d.\vert u_{nx}^{2}\vert^{2}\,\ud x\,\longrightarrow 0,
$$
which leads to
\begin{equation}\label{ux2}
\beta_{n}^{\frac{\gamma}{2}}\|u_{nx}^{2}\|_{L^2(0,1)}\longrightarrow 0.
\end{equation}
Now thanks to \eqref{eq1} and \eqref{ux2}, we obtain
\begin{equation}\label{unx1}
\beta_{n}^{\frac{\gamma}{2}+1}\|u_{nx}^{1}\|_{L^{2}(0,1)}\longrightarrow 0.
\end{equation}
From \eqref{ux2} and \eqref{unx1}, it follows
\begin{equation}\label{Tn}
\beta_{n}^{\frac{\gamma}{2}}\|T_{n}\|_{L^{2}(0,1)} \longrightarrow 0.
\end{equation}
Taking the inner product of \eqref{eq3} with $u_n^2$ in $L^{2}(0,1)$ we get
\begin{equation}\label{SMKV1}
\beta_{n}^{\frac{3\gamma}{4}}\left(i\beta_{n}\|u_{n}^{2}\|^{2}_{L^{2}(0,1)}+\langle T_{n},u_{nx}^{2}\rangle_{L^2(0,1)}+T_{n}(0^{+})\overline{u_{n}^{2}}(0^+)+\langle v_{n}^{2} , u_{n}^{2}\rangle_{L^2(0,1)}\right)=o(1).
\end{equation}
Thanks to \eqref{SMKV4}, \eqref{ux2} and \eqref{Tn}, it's clear that the second and the last terms converge to zero. Furthermore, we have
$$
\beta_{n}^{\frac{3\gamma}{4}}T_{n}(0^{+})\overline{u_n^2}(0^+)\leq C\beta_{n}^{\frac{\gamma}{2}}\left(\|T_{n}\|_{L^2(0,1)}^{\frac{1}{2}}.\|u_{nx}^{2}\|_{L^2(0,1)}^{\frac{1}{2}}.\|T_{n}'\|_{L^{2}(0,1)}^{\frac{1}{2}}.\|u_{n}^{2}\|_{L^2(0,1)}^{\frac{1}{2}}\right).
$$
From \eqref{eq3} we can see that $\|\beta_{n}u_{n}^{2}+v_{n}^{2}\|_{L^2(0,1)}\sim\|T_{n}'\|_{L^2(0,1)}$ which implies that
\begin{align}\label{SMKV2}
\beta_{n}^{\frac{3\gamma}{4}}|T_{n}(0^{+})|.|\overline{u_{n}^{2}}(0^{+})|&\leq C\beta_{n}^{\frac{3\gamma}{4}}\|T_{n}\|_{L^{2}(0,1)}^{\frac{1}{2}}.\|u_{nx}^{2}\|_{L^{2}(0,1)}^{\frac{1}{2}}\times\nonumber
\\
&\left(\|\beta_{n}u_{n}^{2}\|_{L^{2}(0,1)}^{\frac{1}{2}}+\|v_{n}^{2}\|_{L^{2}(0,1)}^{\frac{1}{2}}+o(1)\right).\|u_{n}^{2}\|_{L^{2}(0,1)}^{\frac{1}{2}}\nonumber
\\
&\leq C\|\beta_{n}^{\frac{\gamma}{2}}T_{n}\|_{L^2(0,1)}^{\frac{1}{2}}.\|\beta_{n}^{\frac{\gamma}{2}}u_{nx}^{2}\|_{L^2(0,1)}^{\frac{1}{2}}\times\nonumber
\\
&\left(\|\beta_{n}u_{n}^{2}\|_{L^2(0,1)}^{\frac{1}{2}}+\|v_{n}^{2}\|^{\frac{1}{2}}_{L^2(0,1)}\right)\|\beta_{n}^{\frac{\gamma}{2}}u_{n}^{2}\|_{L^2(0,1)}^{\frac{1}{2}}+o(1)\nonumber
\\
&\leq\left(1+\beta_{n}^{\frac{1}{2}+\frac{\gamma}{4}}.\|u_{n}^{2}\|_{L^2(0,1)}\right)o(1).
\end{align}
Combining \eqref{SMKV1} and \eqref{SMKV2}, one follows
\begin{equation}\label{un2}
\beta_{n}^{\frac{1}{2}+\frac{3\gamma}{8}}\|u_{n}^{2}\|_{L^{2}(0,1)}\,\longrightarrow 0.
\end{equation}
Moreover, multiplying \eqref{eq3} by $\beta_{n}^{-\frac{\gamma}{2}}(1-x)T_{n}$ and integrating over the interval $(0,1)$ then by taking account of \eqref{Tn}, an integration by parts leads to
\begin{align}\label{SMKV5}
\re\langle i \beta_{n}^{\frac{1}{2}+\frac{3\gamma}{8}} u_{n}^{2},(1-x)\beta_{n}^{\frac{1}{2}-\frac{3\gamma}{8}+\frac{\gamma}{2}}T_{n}\rangle_{L^{2}(0,1)}+\frac{\beta_{n}^{\frac{\gamma}{2}}}{2}\left(|T_{n}(0^{+})|^{2}-\|T_{n}\|^2_{L^2(0,1)}\right)
\\
+\beta_{n}^{\frac{\gamma}{2}}\re\langle v_{n}^{2},(1-x)T_{n}\rangle_{L^2(0,1)}=o(1).\nonumber
\end{align}
We suppose that $\ds\gamma\geq\frac{4}{3}$. It's clear from \eqref{SMKV4}, \eqref{Tn} and \eqref{un2} that the first, the third and the last terms of \eqref{SMKV5} converge to zero then one gets
\begin{equation}\label{Tn0}
\beta_{n}^{\frac{\gamma}{4}}.|T_{n}(0^{+})|\,\longrightarrow 0.
\end{equation}
Taking into account to \eqref{unx1} then the trace formula gives
\begin{equation}\label{u10}
\beta_{n}^{\frac{\gamma}{2}+1}.|u_{n}^{1}(0^{+})|\,\longrightarrow 0 .
\end{equation}
Substituting \eqref{eq2} into \eqref{eq3} and taking the inner product with $\beta_{n}^{3-\gamma}v_{n}^{1}$ in $L^{2}(0,1)$ then by integrating by parts we have
\begin{equation}\label{SMKV3}
i\beta_{n}^{4}\left\langle u_{n}^{2},v_{n}^{1}\right\rangle_{L^{2}(0,1)}+\beta_{n}^{3}\left\langle T_{n},v_{n}^{1}\right\rangle_{L^{2}(0,1)}+i\beta_{n}^{4}\|v_{n}^{1}\|_{L^{2}(0,1)}^{2}+\beta_{n}^{3}T_{n}(0^{+})\overline{v_{n}^{1}(0^{+})}=o(1).
\end{equation}
Taking $\ds\gamma\geq12$ and using \eqref{SMKV4}, \eqref{Tn}, \eqref{un2} and \eqref{Tn0} we can see that the first, the second and the fourth terms of \eqref{SMKV3} converge to zero, therefore
\begin{equation}\label{v1}
\beta_{n}^{2}.\|v_{n}^{1}\|_{L^{2}(0,1)}\,\longrightarrow 0.
\end{equation}
From \eqref{eq2} and \eqref{v1} it follows
\begin{equation}\label{v2}
\beta_{n}\|v_{n}^{2}\|_{L^2(0,1)}\,\longrightarrow 0.
\end{equation}
Multiplying \eqref{eq4} with $\beta_{n}^{-\gamma}(1-x)\overline{v_{nx}^{1}}$ and integrating over $(0,1)$ then by taking the real part we find 
\begin{align*}
\frac{c}{2}\left(|v_{nx}^{1}(0^{+})|^{2}-\|v_{nx}^{1}\|_{L^{2}(0,1)}^{2}\right)
=\re\left\langle u_{n}^{2},(1-x)v_{nx}^{1}\right\rangle_{L^{2}(0,1)}
\\
-\re\langle i\beta_{n}v_{n}^{2},(1-x)v_{nx}^{1}\rangle_{L^{2}(0,1)}+o(1).
\end{align*}
Using \eqref{SMKV4}, \eqref{un2} and \eqref{v2} leads to 
\begin{equation}\label{SMKV17}
|v_{nx}^{1}(0^{+})|^{2}-\|v_{nx}^{1}\|_{L^{2}(0,1)}^{2}\,\longrightarrow 0.
\end{equation}
We take the inner product of \eqref{eq4} with $\beta_{n}^{-\gamma}xv_{n}^{1}$ in $L^{2}(0,1)$ then we have 
\begin{equation*}
c\left(\int_{0}^{1}x|v_{nx}^{1}(x)|^{2}\,\ud x+\left\langle v_{nx}^{1},v_{n}^{1}\right\rangle_{L^{2}(0,1)}\right)=\left\langle u_{n}^{2},xv_{n}^{1}\right\rangle_{L^{2}(0,1)}-i\beta_{n}\langle v_{n}^{2},xv_{n}^{1}\rangle_{L^{2}(0,1)}+o(1).
\end{equation*}
Using \eqref{SMKV4}, \eqref{un2} and \eqref{v2} we deduce that
\begin{equation*}
\int_{0}^{1}x|v_{nx}^{1}(x)|^{2}\,\ud x\,\longrightarrow 0.
\end{equation*}
This implies in particular that for every $\varepsilon$ in $(0,1)$ we have 
\begin{equation}\label{SMKV6}
\|v_{nx}^{1}\|_{L^{2}(\varepsilon,1)}\,\longrightarrow 0\;\text{ as }n\nearrow+\infty.
\end{equation}
Multiplying \eqref{eq4} with $\beta_{n}^{-\gamma}(1-x)\overline{v_{nx}^{1}}$ and integrating over $(0,\varepsilon)$ then by taking the real part we find 
\begin{align*}
\frac{c}{2}\left(|v_{nx}^{1}(\varepsilon)|^{2}-\|v_{nx}^{1}\|_{L^{2}(\varepsilon,1)}^{2}\right)=\re\left\langle u_{n}^{2},(1-x)v_{nx}^{1}\right\rangle_{L^{2}(\varepsilon,1)}-\re\langle i\beta_{n}v_{n}^{2},(1-x)v_{nx}^{1}\rangle_{L^{2}(\varepsilon,1)}+o(1).
\end{align*}
Besides,  from \eqref{SMKV4}, \eqref{un2}, \eqref{v2} and \eqref{SMKV6} we follow
\begin{equation*}
|v_{nx}^{1}(\varepsilon)|\,\longrightarrow 0\;\text{ as }n\nearrow+\infty.
\end{equation*}
Then we deduce that 
\begin{equation}\label{SMKV18}
v_{nx}^{1}(x)\,\longrightarrow 0\;\text{ a.e. in } [0,1]\text{ as }n\nearrow+\infty.
\end{equation}
Now, \eqref{SMKV4} and \eqref{SMKV18} allows  the use of the dominated convergence theorem and lead to
\begin{equation}\label{SMKV9}
\|v_{nx}^{1}\|_{L^{2}(0,1)}\,\longrightarrow 0.
\end{equation}
Therefore, we obtain
\begin{equation}\label{SMKV10}
|v_{n}^{1}(0^{+})|\,\longrightarrow 0.
\end{equation}
By combining \eqref{SMKV17} and \eqref{SMKV9} we find
\begin{equation}\label{SMKV8}
|v_{nx}^{1}(0^{+})|\,\longrightarrow 0.
\end{equation}
Furthermore, taking the inner product of \eqref{eq2} with $\beta_n^{1-\gamma}(1-x)v_{nx}^1$ and then considering the imaginary part one gets
\begin{align*}
\beta_n^2 Re(v_{nx}^1 , (1-x)v_n^1)-Im\beta_n(v_n^2, (1-x)v_{nx}^1)=o(1)
\\
=\frac{1}{2}( \beta_n^2 \vert v_n^1(0^+) \vert ^2 - \beta_n^2 \Vert v_n^1 \Vert^2) - \beta_n Im \langle v_n^2 , (1-x) v_{nx}^1\rangle
\end{align*}
Adding to this \eqref{SMKV10}, \eqref{v1} and \eqref{v2} we can deduce that :
\begin{equation}\label{ vn01}
\beta _n  \vert v_n^1 (0^+) \vert\longrightarrow 0
\end{equation}
Thanks to \eqref{Tn0}, \eqref{u10}, \eqref{SMKV10} and \eqref{SMKV8} one gets
\begin{align}
\beta_{n}^{\frac{\gamma}{2}+1}.u_{n}^{1}(0^{-})\,\longrightarrow 0,\label{0-1}
\\
\beta^{\frac{\gamma}{4}}.u_{nx}^{1}(0^{-})\,\longrightarrow 0,
\\
\beta_n v_{n}^{1}(0^{-})\,\longrightarrow 0,
\\
v_{nx}^{1}(0^{-})\,\longrightarrow 0.\label{0-2}
\end{align} 
Next, inserting \eqref{eq1} into \eqref{eq3} and inserting \eqref{eq2} into \eqref{eq4} and consider both equations in the interval $(0,1)$, leads to
\begin{equation}\label{SMKV11}
-\beta_{n}^{2}u_{n}^{1}-u_{nxx}^{1}+v_{n}^{2}=\beta_{n}^{-\gamma}f_{n}^{2}+i\beta_{n}^{1-\gamma}f_{n}^{1},
\end{equation}
 and
\begin{equation}\label{SMKV12}
 -\beta_{n}^{2}v_{n}^{1}-cv_{nxx}^{1}-u_{n}^{2}=\beta_{n}^{-\gamma}g_{n}^{2}+i\beta_{n}^{1-\gamma}g_{n}^{1}.
\end{equation}
A straightforward calculation shows that the real part of the inner product of \eqref{SMKV11} with $(x+1).u_{nx}^{1}$ and that the real part of the inner of \eqref{SMKV12} with $(x+1).v_{nx}^{1}$ leads to
\begin{align}\label{-104}
\frac{1}{2}\int_{-1}^{0}\left(|\beta_{n}u_{n}^{1}|^{2}+|u_{nx}^{1}|^{2}\right)\,\ud x=\frac{1}{2}\left(|u_{nx}^{1}(0^{-})|^{2}+\beta_{n}^{2}|u_{n}^{1}(0^{-})|^{2}\right)
\\
-\re\langle v_{n}^{2},(x+1)u_{nx}^{1}\rangle_{L^{2}(-1,0)}+o(1),\nonumber
\end{align}
and
\begin{align}\label{-105}
\frac{1}{2}\int_{-1}^{0}\left(|\beta_{n}v_{n}^{1}|^{2}+c|v_{nx}^{1}|^{2}\right)\,\ud x=\frac{1}{2}\left(c|v_{nx}^{1}(0^{-})|^{2}+\beta_{n}^{2}|v_{n}^{1}(0^{-})|^{2}\right)
\\
+\re\langle u_{n}^{2},(x+1)v_{nx}^{1}\rangle_{L^{2}(-1,0)}+o(1).\nonumber
\end{align}
Where we have used \eqref{SMKV4}-\eqref{eq4}. In another hand, from \eqref{SMKV4}, \eqref{un2}, \eqref{v2} and \eqref{0-1}-\eqref{0-2} we get 
\begin{equation}\label{SMKV13}
\int_{-1}^{0}\left(|\beta_{n}u_{n}^{1}|^{2}+|u_{nx}^{1}|^{2}\right)\,\ud x\,\longrightarrow 0,
\end{equation}
and
\begin{equation}\label{SMKV14}
\int_{-1}^{0}\left(|\beta_{n}v_{n}^{1}|^{2}+c|v_{nx}^{1}|^{2}\right)\,\ud x\,\longrightarrow 0.
\end{equation}
Now by summing \eqref{unx1} \eqref{un2}, \eqref{v1}, \eqref{v2}, \eqref{SMKV13} and \eqref{SMKV14} we can see that
\begin{equation}\label{SMKV15}  
\|Y_{n}\|_{\mathcal{H}}\,\longrightarrow 0.
\end{equation}
This contradicts \eqref{SMKV4} and so \eqref{SMKV16} holds true with $\gamma\geq12$. This completes the proof.
\end{proof}


\begin{thebibliography}{99}
\bibitem{F.P}  F.~Alabau, P.~Cannarsa, V.~Komornik, Indirect internal stabilization of weakly coupled evolution equations, J. Evol. Equ.2 (2002) 127–150.

\bibitem{AHR2} {\sc K.~Ammari, F.~Hassine and L.~Robbiano}, Stabilization for the wave equation with singular Kelvin-Voigt damping, {\em arXiv:1805.10430.}

\bibitem{ALS} K.~Ammari, Z.~Liu and F.~Shel, Stabilization for the wave equation with singular Kelvin-Voigt damping, {\em arXiv:1805.10430.} 

\bibitem{ammari-niciase} {\sc K.~Ammari and S.~Nicaise,} {\em Stabilization of elastic systems by collocated feedback,} 2124, Springer, Cham, 2015.

\bibitem{BT} A.~Borichev and Y.~Tomilov, Optimal polynomial decay of functions and operator semigroups, {\em Math. Ann.,} {\bf 347} (2010), 455--478.

\bibitem{hassine1}{\sc F.~Hassine,}  Stability of elastic transmission systems with a local Kelvin–Voigt damping, {\em European Journal of Control,} {\bf 23} (2015), 84--93.

\bibitem{hassine2}{\sc F.~Hassine,}  Asymptotic behavior of the transmission Euler-Bernoulli plate and wave equation with a localized Kelvin-Voigt damping, {\em Discrete and Continuous Dynamical Systems - Series B,} {\bf 21} (2016), 1757--1774.

\bibitem{hassine4}{\sc F.~Hassine,}  Energy decay estimates of elastic transmission wave/beam systems with a local Kelvin-Voigt damping, {\em Internat. J. Control,} {\bf 89}(10) (2016), 1933–-1950.

\bibitem{hassine3}{\sc F.~Hassine,}  Logarithmic stabilization of the Euler-Bernoulli plate equation with locally distributed Kelvin-Voigt damping, {\em Evolution Equations and Control Theory,} {\bf 455}(2) (2017), 1765--1782.

\bibitem{huang1} F.~Huang,  Characteristic conditions for exponential stability of linear dynamical systems in Hilbert space, {\em Ann. Differential Equations}, {\bf 1} (1985), 43--56.

\bibitem{huang2} F.~Huang, On the mathematical model for linear elastic systems with analytic damping,
SIAM J. Control Optim., {\bf 26}(3) (1988), 714--724.

\bibitem{liu-liu} {\sc K.~Liu and Z.~Liu,} Exponential decay of energy of the Euler--Bernoulli beam with locally distributed Kelvin--Voigt damping, {\em SIAM Journal on Control and Optimization,} {\bf 36} (1998), 1086--1098.

\bibitem{liu-rao1} {\sc K. S.~Liu and B.~Rao,} Characterization of polynomial decay rate for the solution of linear evolution equation, {\em Zeitschrift f\"ur Angewandte Mathematik und Physik (ZAMP),} {\bf 56} (2005), 630--644.

\bibitem{liu-rao2} {\sc K. S.~Liu and B.~Rao,} Exponential stability for wave equations with local Kelvin-Voigt damping, {\em Zeitschrift f\"ur Angewandte Mathematik und Physik (ZAMP),} {\bf 57} (2006), 419--432.

\bibitem{liu-zhang} Z.~Liu and Q.~Zhang, Stability of a string with local Kelvin-Voigt damping and nonsmooth coefficient at interface, {\em SIAM J. Control Optim.,} {\bf 54} (2016), 1859--1871. 

\bibitem{Pazy} A.~Pazy, {\em Semigroups of linear operators and applications to partial differential equations}, Springer, New York, 1983.

\bibitem{H.P} H.~Portillo Oquendo and P.~Sánez Pacheco, Optimal decay for coupled waves with Kelvin–Voigt damping, {\em Applied Mathematics Letters} {\bf 67} (2017), 16-20.

\bibitem{pruss} J.~Pr\"{u}ss, On the spectrum of $C_0$-semigroups, {\em Trans. Amer. Math. Soc.}, {\bf 248} (1984), 847--857.

\bibitem{tebou} {\sc L.~Tebou,} A constructive method for the stabilization of the wave equation with
localized Kelvin–Voigt damping, {\em C. R. Acad. Sci. Paris}, Ser. I, {\bf 350} (2012), 603--608.

\end{thebibliography}
 \end{document}